\documentclass[12 pt]{article}
\usepackage{amsmath}
\usepackage{amsthm}
\usepackage{graphicx}
\usepackage{bbm,amssymb}
\usepackage[hyperfootnotes=false,colorlinks]{hyperref}

\def\om{\omega}
\def\O{\Omega}
\def\d{\delta}
\def\Dl{\Delta}
\def\ep{\epsilon}

\def\k{\kappa}

\def\f{\phi}

\def\Lm{\Lambda}
\def\l{\lambda}
\def\y{\eta}

\def\all{\forall}

\def\comp{\ensuremath\mathop{\scalebox{.6}{$\circ$}}}
\def\conj{\overline}

\def\oo{\infty}
\def\para{\parallel}
\def\tld{\tilde}
\def\wed{\wedge}

\def\de{\partial}
\def\gr{\nabla}
\def\inv{^{-1}}
\def\sub{\subset}

\newcommand{\rc}[1]{\frac{1}{#1}} 

\newcommand{\R}{{\mathbb{R}}}

\newcommand{\N}{{\mathbb{N}}}

\def\bx{{\mathbf{x}}}
\def\by{{\mathbf{y}}}
\def\ytd{\tilde{\y}}

\newtheorem{lem}{Lemma}
\newtheorem{prop}{Proposition}
\newtheorem{rem}{Remark}
\newtheorem{cor}{Corollary}

\newtheorem{thm}{Theorem}

\title{Locating the first nodal set in higher dimensions}
\author{Fan Zheng \footnote{Massachusetts Institute of Technology}
}

\date{04/17/13}

\begin{document}

\maketitle
\renewcommand{\thefootnote}{}
\renewcommand{\thefootnote}{\arabic{footnote}}

\begin{abstract} 
This paper estimates the location and the width of the nodal set of the first Neumann eigenfunctions on a smooth convex domain $\O \sub \R^n$, whose length is normalized to be 1 and whose cross-section is contained in a ball of radius $\ep$. In \cite{CJK2009}, an $O(\ep)$ bound was obtained by constructing a coordinate system. In this paper, we present a simpler method that does not require such a coordinate system. Moreover, in the special case $n = 2$, we obtain an $O(\ep^2)$ bound on the width of the nodal set, in analogy to the corresponding result in the Dirichlet case obtained in \cite{GJ1995}.



\end{abstract}

\pagebreak
\tableofcontents

\section{Introduction and Statement of Results}
This paper concerns the nodal set of the eigenfunction of the Laplacian. The nodal set is the set of zeros of the eigenfunction. The geometry of the nodal set is greatly affected by the domain on which the Laplacian is defined. In particular, if the domain is long and narrow, then intuition suggests that the nodal set should be concentrated around a hyperplane because there is little room for it to "wiggle around". Moreover, since the domain can be approximated by an inhomogeneous rod, the location of the nodal set can be approximated by the zero of the eigenfunction of a suitable ordinary differential equation. Although such an approximation is intuitively plausible, its validity needs to be rigorously established. The Dirichlet eigenfunction of a planar convex domain was investigated in \cite{J1995} and \cite{GJ1996}. The corresponding Neumann problem (still in 2 dimensions) was studied in \cite{J2000}, using the variation of the eigenvalue of the auxillary ordinary differential equation. That result was extended to $n$ dimensions in \cite{CJK2009}, in which a coordinate systme was constructed to transform the domain to the cylinder $[0, 1] \times B^{n-1}(1)$.

In this paper, we rederive the result in \cite{CJK2009} for the $n$-dimensional Neumann nodal set using a different approach, and improve the estimate when $n=2$. Before laying out our plan, we first need to fix some notations and normalization conventions. We suppose that $\O \sub \R^n$ is a smooth convex domain that spans a length of 1 in one direction but has widths less than $\ep$ in the other $n-1$ directions, i.e.,
\[
\O \sub [0, 1] \times B^{n-1}(\ep)
\]
where $B^{n-1}(\ep)$ is the open $(n-1)$-dimensional ball of radius $\ep$ centered at 0. We will denote a generic point in $\R^n$ by $\bx = (x, \by)$. With this notation, we define the cross-section of $\O$ to be
\[
\O(s) := \O \cap \{x = s\} \neq \varnothing
\]
for $s \in (0, 1)$. Let $\om(s) = |\O(s)|_{\R^{n-1}}$ denote the $(n-1)$-dimensional volume of $\O(s)$.

Let $u$ be an eigenfunction with the smallest nonzero eigenvalue $\l$ of the Neumann problem
\begin{equation}\label{PDE}
\begin{cases}
\Dl u(\bx) = -\l u(\bx),\, \bx \in \O\\
u_\nu(\bx) = 0,\, \bx \in \de \O
\end{cases}
\end{equation}
where $u_\nu$ denotes the normal derivative in the outward direction. We let
\[
\Lm = \{\bx \in \O: u(\bx) = 0\}
\]
denote the nodal set of $u$. In this paper we will provide estimates on both the location (in terms of the projection onto the $x$ coordinate) and the width (in terms of the length of such a projection). Our approach starts by looking at the average of the eigenfuction
\begin{equation}\label{deff}
\bar u(x) = \frac{\int_{\O(x)} u(x,\by)d\by}{\om(x)}
\end{equation}
on the cross-section and deriving an ODE that approximates the behavior of $\f$. As has been observed by \cite{CJK2009}, the energy functional of this ODE resembles the one dimensional Neumann problem:
\begin{equation}\label{ODE}
-(\om\f')' = \mu\om\f
\end{equation}
Therefore, one eigenfunction can be used as a test function of the energy functional related to the other eigenfunction. By the variational principle, the Neumann eigenvalue $\l$ of the Laplacian can be approximated by the correponding eigenvalue $\mu$ of (\ref{ODE}), provided that the error term arising from this approximation is appropriately bounded. One source of the error term comes from the variation of the boundary of the cross-section $\O(x)$, which can be controlled by the convexity of $\O$. Another source of error is the transverse variation of $u$ across $\O(x)$, which we control by using a gradient estimate in \cite{CJK2009}. Now, by standard ODE comparison theorems, $\bar u$ itself can be approximated by the Neumann eigenfunction $\f$, which means that the zero of $\bar u$ is near the zero of $\f$. Since we have already controlled the transverse variation of $u$, the nodal set $\Lm$ can be nailed down with good precision.

In two dimensions, we have more to say. The only ``transverse" variation of $u$ is $\de_y u$, which satisfies the same PDE as $u$ in $\O$. The boundary value of $\de_y u$ is controlled by $\de_x u$ via the Neumann boundary constraint, which in turn must satisfy the gradient bounds. Thererfore, we can apply the maximum principle to $\de_y u$ to obtain a better bound on the transverse variation of $u$, narrowing the difference between $u$ and $\bar u$, and further nailing down the width of $\Lm$.

The difficulties we have experienced when pushing the $O(\ep^2)$ estimates to higher dimensions are those typical of Sobolev-type estimates. The energy constraint imposed by the eigenvalue is on the $L^2$ norm of the variation of $u$ across $\O(x)$, while the width of the nodal set is essentially its $L^\oo$ norm. So far the author is only able to obtain an $O(\ep^{1+\rc{2(n-1)}})$ bound in the $n$ dimensional case, an estimate that deterioriates as $n \to \oo$.

Another difficulty arises when one attempts to drop the smoothness of $\O$. A natural approach is to approximate $\O$ by a sequence of smooth domains $\O_k$, and derive the convergence of the corresponding eigenfunctions via the gradient bound (\ref{grleu}), as sketched in the last paragraph of \cite{CJK2009}. However, in order for this approach to work, one needs to establish the simplicity of the first Neumann eigenvalue of $\O$. This was done in \cite{B1999} in 2 dimensions (see Proposition 2.4). The argument there relies on some information about the direction of the normal vectors of $\O$ that does not easily generalize to higher dimensions. An alternative proposed in \cite{B1999} is to use the gradient bound (\ref{grleu}). However, to generalixe (\ref{grleu}) to non-smooth domains seems to require an approximation that brings us back to the very same problem of the simplicity of the eigenvalue.

Finally we point out that all the constants involved (mostly denoted by $C$, $\d$, etc.) depend on the dimension but not on $\ep$ nor the shape of $\O$. The size of $\O(x)$, characterized by $\ep$, is assumed to be smaller than a fixed constant (which may also depend on the dimension). To avoid the proliferation of the symbol $C$, we shall use Landau's and Vinogradov's notation to write $f \ll g$, $g \gg f$, or $f = O(g)$, for $|f| \le Cg$ for some constant $C$ (that possibly depends on the dimension). If $C$ depends on another variable (say $\d$), then we write $f \ll_\d g$ or $f = O_\d(g)$.

We now state our main results as follows:

\begin{thm}\label{main}
Suppose $u$ is a Neumann eigenfunction of (\ref{PDE}) with the smallest nonzero eigenvalue $\l$; $\bar u$ is the cross-sectional average of $u$ defined by (\ref{deff}); $s_0$ is the smallest zero of $\bar u$. Suppose $\f$ is the Neumann eigenfunction of (\ref{ODE}) with the smallest nonzero eigenvalue $\mu$; $s_1$ is the unique zero of $\f$. $\f$ is normalized so that $\f(0) = \bar u(0)$. Then

(a) $\l \le \mu \le \l + O(\ep)$. (Proposition \ref{mu<l+Ce})

(b) $|\bar u(x) - \f(x)| \ll \ep \sup |u|$ for all $x \in [0, 1]$. (Proposition \ref{f-psi<Ce})

(c) The width of the projection of $\Lm$ in the $x$ direction is $\ll \ep$. (Proposition \ref{x-s1<Ce})

(d) When $n = 2$, the bound in (c) can be improved to $\ll \ep^2$. (Proposition \ref{widthep2})
\end{thm}

\begin{rem}
The first inequality in (a) is a direct consequence of the variation principle, see \cite{CJK2009}, (1.10).
\end{rem}

\begin{rem}
Although (\ref{ODE}) may be singular at the endpoints, in this case the existence of the Neumann eigenfunction $\f$ can be shown using a standard Picard iteration, see Proposition \ref{wellpose} in Appendix A.
\end{rem}

\section{Preliminary Estimates of the PDE}
In this section we record some basic estimates of $u$ from \cite{CJK2009} which we will frequently use later. We will write $\|x\|$ as a short hand for $\min(x, 1-x)$.

By using a sinusoidal test function, we know that the first Neumann eigenvalue $\l$ is bounded above by a constant. See \cite{CJK2009}, Remark 1, also see \cite{B1999}, Propsition 2.2 for a different approach from comparison theorems. A uniform lower bound of $\l$, also mentioned in that remark, can alternatively be obtained from part (a) of Theorem \ref{main}, see (\ref{eigenconst}).

Next we quote (an equivalent form of) the gradient bound proved in \cite{CJK2009}, Corollary 2.4, which will be the starting point of our estimate.

\begin{lem}
For all $(x, \by) \in \O$,
\begin{equation}\label{grleu}
|\gr u(x, \by)| \ll \max(\|x\|, \ep)\sup |u|
\end{equation}
\end{lem}

This result was derived from a weaker bound $|\gr u(x, \by)| \ll \sup |u|$ using an iteration method. The citation of that result, however, is likely to confuse some readers. Here we give a more direct source: \cite{Sperb}, Corollary 5.2, which can be used to deduce \cite{CJK2009}, Corollary 2.3. Taking the geometry of $\O$ into account, we have a more convenient form of (\ref{grleu}).

\begin{cor}
for every pair of points $(x_1, \by_1)$ and $(x_2, \by_2) \in \O$,
\begin{equation}\label{varulexCu}
|u(x_1, \by_1) - u(x_2, \by_2)| \ll (|x_1 - x_2| + \ep) \sup |u|
\end{equation}
\end{cor}

Since $\bar u(x)$ is the cross-sectional average of $u$, the $u(x_i, \by_i)$ in (\ref{varulexCu}) can be replaced by $\bar u(x_i)$.

The following inequality (\cite{CJK2009}, Theorem 2.2) is also useful. \begin{lem} 
\begin{equation}\label{infule-csupu}
\sup u \ll -\inf u \ll \sup u
\end{equation}
\end{lem}

In particular, it allows us to replace the $\sup |u|$ in (\ref{grleu}) and (\ref{varulexCu}) by $\sup u$ or $-\inf u$, with some variation of the corresponding constants.

Suppose $(x, \by) \to (x_0, \by_0)$ when $u \to \sup u$ and $(x, \by) \to (x_1, \by_1)$ when $u \to \inf u$. Then by (\ref{varulexCu}) and (\ref{infule-csupu}), we know that $x_0 \neq x_1$. Without loss of generality we suppose $x_0 < x_1$. Then the following lemma allows us to replace $\sup|u|$ with $\sup |\bar u|$ at the cost of a constant factor.

\begin{lem}(\cite{CJK2009}, (2.6))
\begin{equation}\label{supfgesupu}
\bar u(0+) = u(0+, \by) \ge (1 - O(\ep)) \sup u,\ \bar u(1-) = u(1-, \by) \le (1 - O(\ep)) \inf u
\end{equation}
\end{lem}
Combining (\ref{varulexCu}) and (\ref{supfgesupu}), we get
\begin{cor}
There is a constant $\d_0$ such that:
\begin{equation}\label{|s0|>d}
\bar u(x) \ge \sup u/2,\, \all x \le \d_0,\ \bar u(x) \le \inf u/2,\, \all x \ge 1 - \d_0
\end{equation}
Therefore any zero of $\bar u$ is at least $\d_0$ away from 0 and 1.
\end{cor}
%

\section{Estimates of the Domain $\O$}
This section collects the necessary estimates of the cross-sectional volume $\om$. In a sentence they say that $\om$ can not change much due to the convexity of $\O$.
\begin{lem}\label{omes}
Suppose $0 < a < b < 1$. Then 

(a)
\begin{equation}\label{omaomb}
\om(b) \le \om(a) \left( \frac{b}{a} \right)^{n-1},\, \om(a) \le \om(b) \left( \frac{1-a}{1-b} \right)^{n-1}
\end{equation}

(b)
\begin{align}
\label{omb1} \int_b^1 \om dx &< \frac{1-b}{b-a} \left( \rc{a} \right)^{n-1} \int_a^b \om dx,\ \int_0^a \om dx &< \frac{a}{b-a} \left( \rc{1-b} \right)^{n-1} \int_a^b \om dx
\end{align}
Therefore
\begin{equation}
\label{om01} \int_0^1 \om dx < C(a,b)\int_a^b \om dx
\end{equation}
where
\[
C(a, b) = \rc{b-a} \left[ \left( \rc{a} \right)^{n-1} (1-b) + \left( \rc{1-b} \right)^{n-1} a \right] + 1
\]

(c)
\begin{equation}
\label{omend} \int_0^\ep \om dx + \int_{1-\ep}^1 \om dx \ll \ep \int_\ep^{1-\ep} \om dx
\end{equation}

(d)
\begin{equation}
\label{sup<L2} \sup|\bar u|^2 \int_0^1 \om dx \ll \int_\ep^{1-\ep} \om \bar u^2dx
\end{equation}
\end{lem}

\begin{proof}

(a) This follows from the convexity of $\O$.

(b) Let $\om_m = \min \{\om(x): a \le x \le b\} = \om(x_0)$ for some $x_0 \in [a,b]$. Then by (\ref{omaomb}),
\begin{align*}
\int_b^1 \om dx &\le \om(x_0) \int_b^1 \left( \frac{x}{x_0} \right)^{n-1}dx < \om_m (1-b) \left( \rc{a} \right)^{n-1}\\
&\le \rc{b-a} \left( \rc{a} \right)^{n-1} (1-b) \int_a^b \om dx
\end{align*}
and this proves (\ref{omb1}). The other half follows symmetrically.

(c) Taking $a = 1/2$ and $b = 1 - \ep$ in (\ref{omb1}) we have
\[
\int_{1-\ep}^1 \om dx < \frac{2^{n-1}\ep}{1/2-\ep} \int_{1/2}^{1-\ep} \om dx < 2^{n+1}\ep \int_{1/2}^{1-\ep} \om dx
\]
Adding the symmetric estimate gives (\ref{omend}).

(d) Take $\d_0$ as in (\ref{|s0|>d}). By (\ref{om01}),
\[
\sup |\bar u|^2 \int_0^1 \om dx \ll \sup |\bar u|^2 \int_{\d_0/2}^{\d_0} \om dx \ll \int_{\d_0/2}^{\d_0} \om \bar u^2dx \ll \int_\ep^{1-\ep} \om \bar u^2dx
\]
\end{proof}

\section{Energy Estimates of the Cross-sectional Average $\bar u$}
Now we are prepared to derive the ODE satisfied by $\bar u$ and compare it to (\ref{ODE}). To do so, we apply the divergence theorem to the region $\{s \le x\}$.
\begin{align}
\nonumber \int_{\O(x)} u_x(x,\by)d\by &= \int_{\de\{s \le x\}} u_\nu(s,\by)d\by = \int_{s \le x} \Dl u(s,\by)d\by ds\\
\label{intomux} &= -\l \int_{s \le x} u(s,\by)d\by ds = -\l \int_0^x \om(s)\bar u(s)ds
\end{align}

We define
\begin{equation}\label{defy}
\y(x) = \om(x)\bar u'(x) - \int_{\O(x)} u_x(x,\by)d\by = \om(x)\bar u'(x) + \l \int_0^x \om(s)\bar u(s)ds
\end{equation}
Since $u(x, \by)$ and $\de \O$ are smooth, $\bar u(x)$, hence $\y(x)$, is also smooth, which allows us to differentiate (\ref{defy}) with respect to $x$ to obtain an ODE satisfied by $\bar u$:
\begin{equation}\label{ODEy}
(\om\bar u')' = \y' - \l\om\bar u
\end{equation}
Now we compute the energy of $\bar u$ in terms of its $L^2$ norm (weighted by $\om$) and some error terms.
\begin{align}
\nonumber \int_\ep^{1-\ep} \om\bar u'^2dx &= \bar u\om\bar u'|_\ep^{1-\ep} - \int_\ep^{1-\ep} \bar u(\y'-\l\om\bar u)dx\\
\nonumber &= \bar u(\om\bar u'-\y)|_\ep^{1-\ep} + \int_\ep^{1-\ep} (\bar u'\y+\l\om\bar u^2)dx\\
\nonumber &= \bar u(1-\ep) \int_{\O(1-\ep)} u_x(1-\ep,\by)d\by - \bar u(\ep) \int_{\O(\ep)} u_x(\ep,\by)d\by\\
\label{intomf'2} &+ \int_\ep^{1-\ep} (\bar u'\y+\l\om\bar u^2)dx
\end{align}

The control of the first two terms is easy. By (\ref{intomux}), the uniform bound of $\l$, (\ref{omend}) and (\ref{sup<L2}),
\begin{align}\label{bdterml}
\left | \bar u(\ep)\int_{\O(\ep)} u_x(\ep,\by)d\by\right | &\ll \sup|\bar u|^2\int_0^\ep \om dx \ll \ep\int_\ep^{1-\ep} \om\bar u^2dx
\end{align}
Symmetrically,
\begin{align}\label{bdtermr}
\left | \bar u(1-\ep)\int_{\O(1-\ep)} u_x(1-\ep,\by)d\by\right | \ll \ep\int_\ep^{1-\ep}  \om\bar u^2dx
\end{align}

The bound on the third term in (\ref{intomf'2}) is trickier. First we pick a point $(0, \by_0) \in \de \O \cap \{x = 0\}$ and consider an affine map
\[
T_x: (x, \by) \in \O(x) \to \left( \ep, \by_0 + \frac{\ep}{x} (\by - \by_0) \right)
\]

For $x \in [\ep,1-\ep]$, let $\O_\ep(x) = T_x\O(x)$, and $\om_\ep(x) = |\O_\ep(x)|_{\R^{n-1}}$. Since $T_x$ is affine,
\begin{equation}\label{omep}
\om_\ep(x) = |\det T_x|\om(x) = \left( \frac{\ep}{x} \right)^{n-1} \om(x)
\end{equation}

We now obtain some information on $\O_\ep$ and $\om_\ep$.
\begin{lem}\label{omepdec}

\indent (a) As a set, $\O_\ep(x)$ is monotonely decreasing in $x$.

\indent (b)
\begin{equation}\label{ome'}
-2(n-1) \frac{\om_\ep(x)}{\|x\|} \le \om_\ep'(x) \le 0
\end{equation}
\end{lem}

\begin{proof}
(a) and the second inequality of (b) follows from the convexity of $\O$. Now we show the first inequality. By Brunn-Minkowski Inequality $\om(x)^{1/(n-1)}$ is concave. Therefore,
\begin{align*}
-\frac{\om(x)^{1/(n-1)}}{1-x} &\le \frac{\om(1)^{1/(n-1)} - \om(x)^{1/(n-1)}}{1-x} \le (\om(x)^{1/(n-1)})'\\
&= \frac{\om'(x)}{(n-1)\om(x)^{(n-2)/(n-1)}}
\end{align*}
Hence
\[
-\frac{n-1}{1-x} \le \frac{\om'(x)}{\om(x)} = (\log\om(x))'
\]
From (\ref{omep}) we know $\log\om_\ep(x) = (n-1)(\log \ep - \log x + \log \om(x))$. Then
\[
\frac{\om_\ep'(x)}{\om_\ep(x)} = (\log \om_\ep(x))' = -\frac{n-1}{x} + (\log \om(x))' \ge -2\frac{n-1}{\|x\|}
\]
\end{proof}

Our crucial estimate comes from:
\begin{lem}
For $x \in [\ep,1-\ep]$,
\begin{equation}\label{y<epomf}
|\y(x)| \ll \ep\om(x)\sup|\f|
\end{equation}
\end{lem}

\begin{proof}
Fix an $x_0 \in [\ep, 1 - \ep]$. Let $v = u - \bar u(x_0)$. Then $\gr v = \gr u$. Let
\[
\bar v(x) = \bar u(x) - \bar u(x_0) = \frac{\int_{\O(x)} v(x, \by)d\by}{\om(x)}
\]
Obviuosly $\bar v(x_0) = 0, \bar u'(x) = \bar v'(x)$.

For any $x_1 \in (x_0, 1-\ep)$,
\begin{equation}\label{ftdx1}
\bar v(x_1) = \frac{\int_{\O(x_1)}v} {\int_{\O(x_1)}1} = \frac{\left| \det T_{x_1}\inv \right| \int_{\O_\ep(x_1)} v \comp T_{x_1}\inv}{\left| \det T_{x_1}\inv \right| \int_{\O_\ep(x_1)}1} = \frac{\int_{\O_\ep(x_1)} v \comp T_{x_1}\inv}{\om_\ep(x_1)}
\end{equation}

By Lemma \ref{omepdec} (a), $\O_\ep(x_1) \sub \O_\ep(x_0)$, so we write $\O_\ep(x_0) = \O_\ep(x_1) \cup \Dl\O_\ep(x_1)$.

Since $\bar v(x_0)=0$, $0 = \int_{\O_\ep(x_0)}v\comp T_{x_0}\inv = \int_{\O_\ep(x_1)} v \comp T_{x_0}\inv + \int_{\Dl\O_\ep(x_1)} v \comp T_{x_0}\inv$. Therefore,
\begin{align}
\nonumber \int_{\O_\ep(x_1)} v \comp T_{x_1}\inv &= \int_{\O_\ep(x_1)} (v \comp T_{x_1}\inv - v \comp T_{x_0}\inv) + \int_{\O_\ep(x_1)} v \comp T_{x_0}\inv\\
\label{intuTinv} &=|\det T_{x_0}| \int_{T_{x_0}\inv \O_\ep(x_1)}(v \comp T_{x_1}\inv \comp T_{x_0} - v) - \int_{\Dl\O_\ep(x_1)} v \comp T_{x_0}\inv
\end{align}
Combining (\ref{ftdx1}) and (\ref{intuTinv}) we get
\begin{equation}\label{ftdx12}
\bar v(x_1) = \frac{\det T_{x_0}}{\om_\ep(x_1)} \int_{T_{x_0}\inv \O_\ep(x_1)}(v \comp T_{x_1}\inv \comp T_{x_0} - v) - \rc{\om_\ep(x_1)} \int_{\Dl\O_\ep(x_1)} v \comp T_{x_0}\inv
\end{equation}

A simple calculation yields $T_{x_1}\inv \comp T_{x_0}(x_0, \by) = \left( x_1, \by + \frac{x_1-x_0}{x_0} (\by - \by_0) \right)$. Thus,
\[
\lim_{x_1\to x_0} \frac{v \comp T_{x_1}\inv \comp T_{x_0} - v}{x_1-x_0} = \de_x v(x_0, \by) + \frac{\de_{\by - \by_0} v(x_0, \by)}{x_0} = \de_x u(x_0, \by) + \frac{\de_{\by - \by_0} u(x_0, \by)}{x_0}
\]
Since $\lim_{x_1\to x_0}\O_{\ep}(x_1) = \O_{\ep}(x_0)$, and $\gr v = \gr u$ is bounded by (\ref{grleu}), we can differentiate under the integral sign.
\begin{align}
\nonumber &\lim_{x_1\to x_0} \rc{x_1 - x_0} \frac{\det T_{x_0}}{\om_\ep(x_1)} \int_{T_{x_0}\inv \O_\ep(x_1)}(v \comp T_{x_1}\inv \comp T_{x_0} - v)\\
\nonumber &=\frac{\det T_{x_0}}{\om_\ep(x_0)} \int_{\O(x_0)} \left( \de_x + \frac{\de_{\by-\by_0}}{x_0} \right) u(x_0, \by)d\by\\
\label{uTinvT-u2} &=\rc{\om(x_0)} \int_{\O(x_0)} \left( \de_x + \frac{\de_{\by-\by_0}}{x_0} \right) u(x_0,\by)d\by
\end{align}
where the last equality follows from (\ref{omep}). Therefore, from (\ref{ftdx12}), (\ref{uTinvT-u2}) and $v(x_0) = 0$, we get
\begin{align}
\nonumber |\bar v'(x_0)| &= \lim_{x_1\to x_0} \left| \frac{\bar v(x_1)}{x_1 - x_0} - \rc{\om(x_0)} \int_{\O(x_0)} \de_x u(x_0, \by)d\by \right|\\
\nonumber &\le \left| \rc{\om(x_0)} \int_{\O(x_0)} \frac{\de_{\by - \by_0} u(x_0, \by)}{x_0}d\by \right| + \lim_{x_1\to x_0}\rc{(x_1 - x_0)\om_\ep(x_1)} \left| \int_{\Dl\O_\ep(x_1)} v \comp T_{x_0}\inv \right|\\
\label{ftd-dexu} &\le \left| \rc{\om(x_0)} \int_{\O(x_0)} \frac{\de_{\by - \by_0} u(x_0, \by)}{x_0}d\by \right| + \rc{\om_\ep(x_0)} \lim_{x_1\to x_0} \rc{x_1-x_0} \left| \int_{\Dl\O_\ep(x_1)} v \comp T_{x_0}\inv \right|
\end{align}
By (\ref{grleu}),
\begin{align}
\nonumber \left| \rc{\om(x_0)} \int_{\O(x_0)} \frac{\de_{\by - \by_0} u(x_0, \by)}{x_0}d\by \right| &\le \rc{\|x_0\|} \sup_{\O(x_0)} |\by - \by_0| \sup_{\O(x_0)}|\gr u(x_0, \by)|\\
\label{deyu} &\ll \ep\sup|\bar u|
\end{align}

On the other hand, by (\ref{ome'}) and (\ref{varulexCu}),
\begin{align}
\nonumber \lim_{x_1 \to x_0} \rc{x_1 - x_0} \left| \int_{\Dl\O_\ep(x_1)} v \comp T_{x_0}\inv \right|  &\le \lim_{x_1 \to x_0} \rc{x_1 - x_0} |\Dl\O_\ep(x_1)|_{\R^{n-1}} \sup_{\Dl\O_\ep(x_1)} |v \comp T_{x_0}\inv|\\
\nonumber &= \om_\ep'(x_0) \sup_{\O_\ep(x_0)} |v \comp T_{x_0}\inv| \ll \frac{\om_\ep(x_0)}{\|x_0\|} \sup_{\O(x_0)} |v|\\
\nonumber &= \frac{\om_\ep(x_0)}{\|x_0\|} \sup_{\O(x_0)} |u(x_0, \by) - \bar u(x_0)|\\
\label{intuTinv2} &\ll \om_\ep(x_0) \ep \|x_0\| \sup |\bar u|
\end{align}

Combining (\ref{ftd-dexu}), (\ref{deyu}) and (\ref{intuTinv2}), together with $\bar v' = \bar u'$, we conclude that, for any $x_0 \in [\ep, 1-\ep]$,
\[
\left| \bar u'(x_0) - \rc{\om(x_0)} \int_{\O(x_0)} \de_x u(x_0, \by)d\by \right| \ll \ep\sup|\bar u|
\]
and hence $|\y(x_0)| \ll \ep \om(x_0) \sup |\bar u|$.
\end{proof}

Using these estimates, we now obtain the bound on the energy of $\bar u$.

\begin{lem}
\begin{equation}\label{f'<f}
\int_\ep^{1-\ep} \om \bar u'^2dx \le (\l + O(\ep)) \int_\ep^{1-\ep} \om \bar u^2dx
\end{equation}
\end{lem}

\begin{proof}
Recall that from (\ref{intomf'2}), (\ref{bdterml}) and (\ref{bdtermr}) we know
\[
\int_\ep^{1-\ep} \om \bar u'^2dx \le (\l + O(\ep)) \int_\ep^{1-\ep} \om \bar u^2dx + \int_\ep^{1-\ep} \bar u'\y dx
\]
By (\ref{y<epomf}), Cauchy-Schwarz Inequality and (\ref{sup<L2}),
\begin{align*}
\int_\ep^{1-\ep} \bar u'\y dx &\ll \ep \int_\ep^{1-\ep} \bar u' \om \sup|\bar u|dx \ll \ep \sqrt{\int_\ep^{1-\ep} \om \bar u'^2dx} \sqrt{\int_\ep^{1-\ep} \om \sup|\bar u|^2dx}\\
&\ll \ep \sqrt{\int_\ep^{1-\ep} \om \bar u'^2dx} \sqrt{\int_\ep^{1-\ep} \om \bar u^2dx}\\
\end{align*}
The result follows by solving a quadratic inequality.
\end{proof}

\section{Proof of Theorem \ref{main}, Part (a)}\label{section4}
Since the estimate in (\ref{y<epomf}) only works on $[\ep, 1-\ep]$, we need to cut $\bar u$ off outside this region. Let
\[
\bar u_1(x) =
\begin{cases}
\bar u(\ep), \text{if } x < \ep\\
\bar u(x), \text{if } \ep \le x \le 1-\ep\\
\bar u(1-\ep), \text{if } x > 1-\ep\\
\end{cases}
\]
Since $\bar u_1(x)$ may no longer be orthogonal to the constant function, we adjust it by setting
\begin{equation}\label{cutshift}
\bar u_1^0 = \frac{\int_0^1 \om \bar u_1dx}{\int_0^1 \om dx},\
\bar u_1^\perp = \bar u_1 - \bar u_1^0
\end{equation}

We first show that we have not cut off too much:
\begin{lem}
\begin{equation}\label{f1p>f}
\int_0^1 \om (\bar u_1^\perp)^2dx \ge (1 - O(\ep)) \int_0^1 \om \bar u^2dx
\end{equation}
\end{lem}

\begin{proof}
Since $\bar u_1 = \bar u$ on $[\ep, 1-\ep]$, by (\ref{omend}) and (\ref{sup<L2}),
\[
\left| \int_0^1 \om(\bar u_1^2 - \bar u^2)dx \right| \ll \sup|\bar u|^2 \int_{\|x\| \le \ep} \om dx\ll \ep \int_0^1 \om \bar u^2dx
\]
so
\begin{equation}\label{f1>f}
\int_0^1 \om \bar u_1^2dx \ge (1 - O(\ep)) \int_0^1 \om \bar u^2dx
\end{equation}

Since $u$ is the first non-constant Neumann eigenfunction on $\O$, $u$ is orthogonal to the constant function.
\[
0 = \int_\O u(x, \by)d\by dx = \int_0^1 \om\bar udx
\]
Therefore, by (\ref{omend}),
\[
\left| \int_0^1 \om \bar u_1dx \right| = \left| \int_0^1 \om (\bar u_1 - \bar u)dx \right| \le \int_{\|x\| \le \ep} \om|\bar u_1 - \bar u|dx \ll \ep \sup |\bar u| \int_0^1 \om dx
\]
Thus $\bar u_1^0 \ll \ep \sup |\bar u|$, so by (\ref{sup<L2}) and (\ref{f1>f}),
\begin{align*}
\int_0^1 \om (\bar u_1^\perp)^2dx &= \int_0^1 \om \bar u_1^2dx - \left( \bar u_1^0 \right)^2 \int_0^1 \om dx \ge \int_0^1 \om \bar u_1^2dx - O(\ep^2) \int_0^1 \om \bar u^2dx\\
&\ge (1 - O(\ep)) \int_0^1 \om\bar u^2dx
\end{align*}
\end{proof}

\begin{lem}
\begin{equation}\label{f1p'<f1p}
\int_0^1 \om (\bar u_1^\perp{}')^2dx \le (\l + O(\ep)) \int_0^1 \om (\bar u_1^\perp)^2dx
\end{equation}
\end{lem}

\begin{proof}
Since $\bar u_1^0$ is a constant, and $\bar u_1$ is constant on $\O \cap \{\|x\| \le \ep\}$,
\[
\int_0^1 \om (\bar u_1^\perp{}')^2dx = \int_0^1 \om (\bar u_1')^2dx = \int_\ep^{1-\ep} \om (\bar u_1')^2dx
\]
Therefore, by (\ref{f'<f}), (\ref{f1p>f}) and the uniform bound of $\l$,
\begin{align*}
\int_0^1 \om (\bar u_1^\perp{}')^2dx &\le (\l + O(\ep)) \int_0^1 \om \bar u_1^2dx
\le \frac{\l + O(\ep)}{1 - O(\ep)} \int_0^1 \om \bar u_1^2dx\\
&\le (\l + O(\ep)) \int_0^1 \om (\bar u_1^\perp)^2dx
\end{align*}
\end{proof}

\begin{prop}\label{mu<l+Ce}
$\l \le \mu \le \l + O(\ep)$.
\end{prop}

\begin{proof}
By design (\ref{cutshift}), $\bar u_1^\perp$ is orthogonal to the constant function, so is an valid test function for the Rayleigh quotient of the ODE (\ref{ODE}). By (\ref{f1p'<f1p}),
\[
\mu \le \frac{\int_0^1 \om (\bar u_1^\perp)'^2dx} {\int_0^1 \om (\bar u_1^\perp)^2dx} \le \l + O(\ep)
\]
\end{proof}

Lemma 3.1 in \cite{CJK2009} provides a constant lower bound for $\mu$. Together with Proposition \ref{mu<l+Ce}, it implies
\begin{equation}\label{eigenconst}
1 \ll c - O(\ep) \le \mu - O(\ep) \le \l \le \mu \le \l + O(\ep) \ll 1
\end{equation}

\section{Estimates of the ODE (\ref{ODE})}
In this section we make preparations for the ODE comparison estimates in the next section. Note that by Theorem \ref{wellpose} in Appendix A, there is a unique solution to (\ref{ODE}) with Neumann boundary conditions at 0 and 1 in the classical sense, i.e. we can talk about $\f(0)$, $\f'(0)$, $\f(1)$ and $\f'(1)$, and can normalize $\f$ so that $\f(0) = \bar u(0+) > 0$.

First note that $\f$ is monotonely decreasing because
\[
\f'(x) = 
\begin{cases}
-\rc{\om(x)} \int_0^x \mu \om(s) \f(s)ds,\ x \le s_1\\
\rc{\om(x)} \int_x^1 \mu \om(s) \f(s)ds,\ x \ge s_1
\end{cases} \le 0
\]

Next we show an analog of (\ref{|s0|>d}).
\begin{lem}
There is a constant $\d_1 > 0$ such that
\begin{equation}\label{pd>cp0}
\f(x) \ge \f(0)/2,\ \all x \le \d_1,\, \f(x) \le -\inf \f(1)/2,\ \all x \ge 1 - \d_1
\end{equation}
\end{lem}
\begin{proof}
Since
$\f$ is monotonely decreasing, by (\ref{omaomb}), for all $x \le s_1$,
\[
\f'(x) = -\rc{\om(x)} \int_0^x \mu \om(s) \f(s)ds \ge -\mu x \left( \rc{1-x} \right)^{n-1} \f(0)
\]
By (\ref{eigenconst}), there is a constant $\d_1 > 0$ such that for all $x \le \min(\d_1, s_1)$ we have
\[
\f(0) - \f(x) = -\int_0^x \f'(s)ds \le \mu x \sup_{s \le x} s \left( \rc{1-x} \right)^{n-1} \f(0) \le \frac{\f(0)}{2}
\]
In particular, $\f(\min(\d_1, s_1)) > 0$, which implies that $s_1 > \d_1$, as well as the first half of (\ref{pd>cp0}). The second half follows symmetrically.
\end{proof}

Then we give a ``reverse gradient estimate" for $\f$, which shows that $\f$ is decreasing fast enough in the middle. 
\begin{lem}
\begin{equation}\label{p'<-cp}
-\f'(x) \gg_\d \f(0),\, \all \|x\| \ge \d
\end{equation}
\end{lem}
\begin{proof}
Since $\f$ is the first non-constant eigenfunction of the ODE (\ref{ODE}), it is orthogonal to the constant function. Combining this with the monotocity of $\f$, (\ref{pd>cp0}) and (\ref{om01}), we get
\[
-\f(1) \int_0^1 \om dx > -\int_{s_1}^1 \om  \f dx = \int_0^{s_1} \om \f dx > \frac{\f(0)}{2} \int_{\d_1/2}^{\d_1} \om dx \gg \f(0) \int_0^1 \om dx
\]
so $\f(0) \ll -\f(1)$. A symmetric argument then gives
\begin{equation}\label{infp<-csupp}
\f(0) \ll -\f(1) \ll \f(0)
\end{equation}

Note that $\om(x) \f'(x) = -\mu \int_0^x \om(s) \f(s)ds$. If $x \in [\d, s_1]$, then by (\ref{pd>cp0}), (\ref{eigenconst}) and (\ref{omaomb}),
\[
-\om(x) \f'(x) \gg \f(0) \int_0^{\d_1 \wed \d} \om(s)ds \ge \om(x) \f(0) \int_0^{\d_1 \wed \d} \left( \frac{s}{x} \right)^{n-1}ds \gg_\d \om(x) \f(0)
\]
A symmetric argument, together with (\ref{infp<-csupp}), gives (\ref{p'<-cp}) for $x \in [s_1, 1 - \d]$.
\end{proof}

On the other hand, $\f$ does not change much on the ends.
\begin{lem}
For all $\|x\| \le \ep$,
\begin{equation}\label{dp<Ce2}
|\f(x_1) - \f(x_2)| \ll \ep^2 \f(0) = \ep^2 \bar u(0+)
\end{equation}
\end{lem}

\begin{proof}
By symmetry and (\ref{infp<-csupp}) we only show the case when $x \le \ep$.
By (\ref{p'leCxp}) in Appendix and (\ref{eigenconst}),
\[
|\f'(x)| \le \mu x \sup_{s \le x} |\f(s)| \ll \ep \f(0)
\]
so
\[
|\f(x_1) - \f(x_2)| \le \int_{x_1}^{x_2} |\f'(s)|ds \ll \ep^2 \f(0)
\]
\end{proof}

Equation (\ref{dp<Ce2}) allows us to connect $\f$ with $\bar u$, at least one one end.
\begin{cor}
For all $x \le \ep$,
\begin{equation}\label{f-p<Ce2}
|\bar u(x) - \f(x)| \ll \ep^2 \sup|\bar u|
\end{equation}
\end{cor}

\begin{proof}
Combine the normalization $\f(0) = \bar u(0+)$, (\ref{dp<Ce2}) and (\ref{varulexCu}).
\end{proof}

\section{Proof of Theorem \ref{main}, Parts (b) and (c)}
We normalize $\f$ so that $\f(0) = \bar u(0+)$ as before. The following proposition makes it possible approximate $\bar u$ (and even $u$) by $\f$.

\begin{prop}\label{f-psi<Ce}
For all $x \in [0, 1]$, $|\bar u(x) - \f(x)| \ll \ep \sup|\bar u|$.
\end{prop}

\begin{proof}
Recall that $\bar u$ satisfies
\[
\om(x) \bar u'(x) = \y(x) - \int_0^x \l \om(s) \bar u(s) ds = \y(x) + \int_x^1 \l \om(s) \bar u(s) ds
\]
because of $\bar u$ has mean zero. Similarly, $\f(x)$ satisfies
\[
\om(x) \f'(x) =  -\int_0^x \mu \om(s) \f(s) ds = \int_x^1 \mu \om(s) \bar u(s) ds
\]

We define $\k(x) = \bar u(x) - \f(x)$, and
\[
\ytd(x) = \y(x) - \int_0^x (\l - \mu) \om(s) \bar u(s) ds = \y(x) + \int_x^1 (\l - \mu) \om(s) \bar u(s) ds
\]
Then
\begin{equation}\label{ODEk}
\om(x) |\k(x)|' \le |\om(x) \k'(x)| \le |\ytd(x)| + \min\left( \int_0^x \mu \om(s) |\k(s)|ds, \int_x^1 \mu \om(s) |\k(s)|ds \right)
\end{equation}

By (\ref{y<epomf}), (\ref{eigenconst}) and (\ref{omaomb}), for all $x \in [\ep, 1 - \ep]$,
\begin{align}
\nonumber |\ytd(x)| &\le |\y(x)| + \min\left( \left| \int_0^x (\l - \mu) \om(s) \bar u(s)ds \right|, \left| \int_x^1 (\l - \mu) \om(s) \bar u(s)ds \right| \right)\\
\nonumber &\ll \ep \om(x) \sup |\bar u| + \ep \sup |\bar u| \om(x) \min\left( \int_0^x \left( \frac{1 - s}{1 - x} \right)^{n-1}ds, \int_x^1 \left( \frac{s}{x} \right)^{n-1}ds \right)\\
\label{ytd<epomf} &\ll \ep \om(x) \sup |\bar u|
\end{align}

Let $L(x) = \sup_{s \le x} |\k(s)|$. Pick $\d > 0$. Combining (\ref{ODEk}), (\ref{ytd<epomf}), (\ref{eigenconst}) and (\ref{omaomb}), we get, for all $x \in [\ep, 1 - \d]$,
\begin{align*}
|\om(x) \k'(x)| \ll \ep \om(x) \sup |\bar u| + L(x) \int_0^x \om(s)ds \ll_\d \om(x) (\ep \sup |\bar u| + L(x))
\end{align*}
Hence
\[
L'(x) \le |\k'(x)| \ll_\d \ep \sup |\bar u| + L(x)
\]
By (\ref{f-p<Ce2}), $L(\ep) \ll \ep \sup |\bar u|$. Now Gronwall's inequality gives, for all $x \in [\ep, 1 - \d]$,
\begin{equation}\label{f-p<dCe}
|\bar u(x) - \f(x)| = |\k(x)| \le L(x) \ll_\d \ep \sup |\bar u|
\end{equation}
Combining (\ref{f-p<dCe}) with (\ref{f-p<Ce2}), we know that (\ref{f-p<dCe}) holds for all $x \le 1 - \d$.

Turning to the other end, we let $R(x) = \sup_{x \le s \le 1 - \ep}|\k(s)|$ for $x \le 1 - \ep$. Similarly we get
\[
R'(x) \ll \ep \sup |\bar u| + R(x)
\]
Gronwall's inequality gives
\begin{equation}\label{RleCdk1}
R(1 - \d) \le C(\d) (\ep \sup|\bar u| + R(1 - \ep)) = C(\d)(\ep \sup|\bar u| + \k(1 - \ep))
\end{equation}
where $C(\d)$ is monotonely increasing in $\d$. Then
\begin{align*}
\k(1 - \d) &= \k(1 - \ep) - \int_{1 - \d}^{1 - \ep} \k'(s)ds \ge \k(1 - \ep) - \d R(1 - \d)\\
&\ge (1 - \d C(\d)) \k(1 - \ep) - C(\d) \ep \sup|\bar u|
\end{align*}
Therefore, we can pick $\d_1$ such that $\d_1 C(\d_1) < 1/2$, which allows us conclude
\begin{equation}\label{k1leCsupf}
\k(1 - \ep) \ll \k(1 - \d_1) + \ep \sup|\bar u| \le L(1 - \d_1) + \ep \sup|\bar u| \ll \ep \sup|\bar u|
\end{equation}
Combining (\ref{RleCdk1}) and (\ref{k1leCsupf}), we get, for all $x \in [1 - \d_1, 1 - \ep]$
\[
|\bar u(x) - \f(x)| = |\k(x)| \le R(1 - \d_1) \ll \ep \sup|\bar u|
\]
so Proposition \ref{f-psi<Ce} holds on $[1 - \d_1, 1 - \ep]$. We can extend its validity to $[0, 1 - \ep]$ by taking $\d = \d_1$ in (\ref{f-p<dCe}). Finally, Combine (\ref{dp<Ce2}) and (\ref{varulexCu}) to cover the whole interval [0, 1].
\end{proof}

\begin{prop}\label{x-s1<Ce}
$(x, \by) \in \Lm$ implies $|x - s_1| \ll \ep$.
\end{prop}

\begin{proof}
Suppose $u(x, \by) = 0$. Then by (\ref{varulexCu}), $|\bar u(x)| \ll \ep \sup |\bar u|$. By Proposition \ref{f-psi<Ce}, $|\f(x)| \le |\bar u(x)| + |\bar u(x) - \f(x)| \ll \ep \sup |\bar u|$. By (\ref{|s0|>d}), $\|x\| > \d_0$. By (\ref{pd>cp0}), $\|s_1\| > \d_1$. Now the implicit constant in (\ref{p'<-cp}) is fixed, and together with (\ref{supfgesupu}) it gives
\[
|x - s_1| \ll \rc{\f(0)} \left| \int_{s_1}^x \f'(s)ds \right| \ll \frac{|\f(x)|}{\f(0)} \ll \frac{\ep \sup |\bar u|}{\bar u(0)} \ll \ep
\]
\end{proof}

For future benefits we prove a ``reversed gradient estimate" (analogous to \ref{p'<-cp}) for $\bar u$.

\begin{lem}
For all $\|x\| \ge \d$,
\begin{equation}\label{f'<-cf}
-\bar u'(x) \gg_\d \sup |\bar u|
\end{equation}
\end{lem}

\begin{proof}
By (\ref{defy}) we have $\om(x) \bar u'(x) = \y(x) - \l \int_0^x \om(s) \bar u(s)ds$. If $x \in [\d, s_0 + C\ep]$, then by (\ref{y<epomf}), (\ref{|s0|>d}), (\ref{omaomb}), and (\ref{eigenconst}), we have
\begin{align*}
-\om(x) \bar u'(x) &\ge \l \frac{\sup |\bar u|}{2} \int_0^{\d_0 \wed \d} \om(s)ds - O(\ep) \om(x) \sup |\bar u| - \l \sup |\bar u| \int_{s_0}^{s_0 + C\ep} \om(s)ds\\
&\ge \l \sup |\bar u| \left( \rc{2} \int_0^{\d_0 \wed \d} - \int_{s_0}^{s_0 + C\ep} \right) \left( \frac{s}{x} \right)^{n-1} ds - O(\ep) \om(x) \sup |\bar u|\\
&\ge \om(x) \sup |\bar u| (c_\d - C\ep - O(\ep)) \gg_{\d,C} \om(x) \sup |\bar u|
\end{align*}

Symmetrically, if we let $s_0'$ be the largest zero of $\bar u$, then we have $-\bar u'(x) \gg_{\d,C} \sup |\bar u|$ if $x \in [s_0' - C\ep, 1 - \d]$. Since $\bar u(s_0) = \bar u(s_0') = 0$, there are points $(s_0, \by)$, $(s_0', \by') \in \Lm$. By Proposition \ref{x-s1<Ce} we have $|s_0' - s_0| \ll \ep$, so there is a constant $C_0$ such that $s_0' - C_0\ep \le s_0 + C_0\ep$. Then the dependence on $C$ can be dropped and (\ref{f'<-cf}) holds for $x \in [\d, s_0 + C_0\ep] \cup [s_0' - C_0\ep, \d] = [\d, 1 - \d]$.
\end{proof}

\section{Proof of Theorem \ref{main}, Part (d)}
In this section we assume $n = 2$. In doing so we can have better estimates of $\de_y u$ using the generalized maximum principle, refining (\ref{grleu}).

\begin{lem}
For all $(x, y) \in \de \O$
\begin{equation}\label{gryleepbd}
|\de_y u(x, y)| \ll \ep \sup |\bar u|
\end{equation}
\end{lem}

\begin{proof}
When $\|x\| \le \ep$, (\ref{gryleepbd}) follows from (\ref{grleu}) and $|\de_y u(x, y)| \le |\gr u(x, y)|$, so we assume $\|x\| > \ep$. Let $p = (x, y) \in \de \O$, $p_0 = (0, y_0) \in \de \O \cap \{x = 0\}$, $p_1 = (1, y_1) \in \de \O \cap \{x = 1\}$. Let $l$ be a line tangent to $\conj \O$ at $p$ ($l$ need not be unique if $p$ happens to be a corner.) Since $\O$ is convex, $p_0$ and $p_1$ must lie on the same side of $l$. WLOG suppose they lie below $l$. Then its slope
\[
k \in \left[\frac{y_1 - y}{x_1 - x}, \frac{y - y_0}{x - x_0} \right] \sub \left[ \frac{-2\ep}{\|x\|}, \frac{2\ep}{\|x\|} \right]
\]
Since $\de_n u(x, y) = 0$, $\gr u(x, y) \para l$. Therefore by (\ref{grleu}),
\[
|\de_y u(x, y)| = k|\de_x u(x, y)| \ll \frac{\ep}{\|x\|} |\gr u(x, y)| \ll \ep \sup |\bar u|
\]

\end{proof}

Now we pass from the boundary to the interior:

\begin{lem}\label{gryleep}
For all $(x, y) \in \O$, $|\de_y u(x, y)| \ll \ep \sup |\bar u|$
\end{lem}

\begin{proof}
Since $\Dl u = -\l u$, $\Dl \de_y u = -\l \de_y u$. We let $w(x, y) = 1 + \cos(y/\ep)$. Since $|y/\ep| \le 1$, $w(x, y) > 0$ on $\conj \O$. Moreover, $(\Dl + \l)w = 1 + (\l - 1/\ep^2) \cos(y/\ep) \le 1 + (\l - 1/\ep^2) \cos 1 < 0$ if $\ep$ is small enough. By the generalized maximum principle (\cite{Protter}, Chapter 2, Section 5, Theorem 11) and (\ref{gryleepbd}) we have
\[
\max_{\conj \O} |\de_y u(x, y)/w(x, y)| = \max_{\de \O} |\de_y u(x, y)/w(x, y)| \ll \ep \sup |\bar u|
\]
Since $w(x, y) \le 2$, $|\de_y u(x, y)| \ll \ep \sup |\bar u|$ in $\O$.
\end{proof}

\begin{cor}
For every pair of points $(x, y_1), (x, y_2) \in \O$,
\begin{equation}\label{secvarep2}
|u(x, y_1) - u(x, y_2)| \ll \ep^2 \sup |\bar u|
\end{equation}
\end{cor}

\begin{prop}\label{widthep2}
When $n = 2$, $(x, y) \in \Lm$ implies $|x - s_0| \ll \ep^2$
\end{prop}

\begin{proof}
Suppose $u(x, y) = 0$. By (\ref{secvarep2}), $|u(x, y_1)| \ll \ep^2 \sup |\bar u|$, so $|\bar u(x)| \ll \ep^2 \sup |\bar u|$, and by (\ref{f'<-cf}),
\[
|x - s_0| < \rc{c_{\d_0} \sup |\bar u|} \left| \int_{s_0}^x \bar u(s)ds \right| \ll \frac{|\bar u(x)|}{\sup |\bar u|} \ll \ep^2
\]
\end{proof}

\section{Appendix A  Well-posedness of the Neumann Problem (\ref{ODE})}
This section concerns the well-posedness of the Neumann eigenfunction problem of (\ref{ODE}). First we prove an existence and uniqueness result of (\ref{ODE}) with Neumann boundary condition at its (possibly) singular endpoints. Then we show the monotocity of the zero of the solution with respect to the parameter $\mu$. Finally we piece together two solutions into an eigenfunction with Neumann boundary conditions on both ends.

\begin{lem}\label{ODExist}
Suppose $\om$ is continuous on $[0, 1]$, $\om$ positive on $(0, 1)$. Suppose either of the following is true:

\indent (a) $\om(0) > 0$.

\indent (b) $\om(0) = 0$ but $\om(x)$ is increasing in $x$ for some short interval $[0, \d]$.

Then (\ref{ODE}) has a unique solution $\f \in C^1[0, 1)$ such that $\f(0) = 1$ and $\f'(0) = 0$. Moreover, this solution is continuous with respect to $\mu$.
\end{lem}

\begin{proof}
Case (a) follows from the usual existence and uniqueness theorem. We now use Picard's iteration to prove Case (b). Consider a systen of integral equations relating $\f$ and $\f_1$:
\[
\begin{cases}
\f(x) = 1 + \int_0^x \f_1(s)ds\\
\f_1(x) = -\rc{\om(x)} \int_0^x \mu \om(s) \f(s)ds
\end{cases}
\]

Suppose $\f^0(x) = 1$, $\f_1^0(x) = 0$, and inductively define
\[
\begin{cases}
\f^{k+1}(x) = 1 + \int_0^x \f_1^k(s)ds\\
\f_1^{k+1}(x) = -\rc{\om(x)} \int_0^x \mu \om(s) \f^k(s)ds
\end{cases}
\]
for all $k \in \N$. Then we have
\[
\begin{cases}
|\f^{k+2}(x) - \f^{k+1}(x)| = \left| \int_0^x (\f_1^{k+1}(s) - \f_1^k(s))ds \right| \le x \sup_{s \le x} |\f_1^{k+1}(s) - \f_1^k(s)|\\
|\f_1^{k+2}(x) - \f_1^{k+1}(x)| = \rc{\om(x)} \left| \int_0^x (\f^{k+1}(s) - \f^k(s))ds \right| \le \mu x \sup_{s \le x} |\f^{k+1}(s) - \f^k(s)|
\end{cases}
\]
Thus
\[
\begin{cases}
\sup_{s \le x} |\f^{k+3}(x) - \f^{k+2}(x)| \le \mu x^2 \sup_{s \le x} |\f^{k+1}(x) - \f^{k}(x)|\\
\sup_{s \le x} |\f_1^{k+3}(x) - \f_1^{k+2}(x)| \le \mu x^2 \sup_{s \le x} |\f_1^{k+1}(x) - \f_1^{k}(x)|
\end{cases}
\]
Therefore $\f^k(x)$ and $\f_1^k(x)$ converge uniformly to $\f(x)$ and $\f_1(x)$ on $[0, \d']$, where $0 < \d' < \min(\d, \sqrt \mu)$, so
\[
\begin{cases}
\f(x) = 1 + \int_0^x \f_1(s)ds\\
\f_1(x) = -\rc{\om(x)} \int_0^x \mu \om(s) \f(s)ds
\end{cases}
\]
Hence $\f(0) = 1$, $\f'(x) = \f_1(x)$, $\f'(0) = \lim_{x \to 0} \f_1(x) = 0$ and $(\om(x) \f_1(x))' = -\mu \om(x) \f(x)$, so $\f \in C^1[0, \d']$ and solves (\ref{ODE}) with the Neumann boundary condition at 0.

Now we turn to uniqueness. we need only to prove it for the zero initial data, namely $\f(0) = \f'(0) = 0$. Suppose we have a solution $\f \in C^1[0, \d']$ with zero initial data. Integrating (\ref{ODE}) from $s = 0$ to $x$ gives: for all $x \in [0, \d']$
\begin{equation}\label{p'leCxp}
|\f'(x)| = \rc{\om(x)} \left| \int_0^x \mu \om(s) \f(s)ds \right| \le \mu \int_0^x |\f(s)|ds \le \mu\ x \sup_{s \le x} |\f(s)|
\end{equation}
Therefore
\[
|\f(x)| = |\int_0^x \f'(s)ds| \le \mu x^2 \sup_{s \le x} |\f(s)|
\]
so
\[
\sup_{s \le x} |\f(s)| < \mu x^2 \sup_{s \le x} |\f(s)|
\]
Since $\mu x^2 < 1$, $\sup_{s \le x} |\f(s)| = 0$, so $\f = 0$ on $[0, x]$. Letting $x$ range through $[0, \d']$ we get $\f = 0$ on $[0, \d']$, and the solution is unique on $[0, \d']$.

Since $\f^k(x)$ is continuous with respect to $\mu$ by way of construction and $\f^k(x) \to \f(x)$ uniformly on $[0, \d']$, $\f(x)$ is continuous with respect to $\mu$ when $x \in [0, \d']$.

Finally, since the ODE is regular away from 0 and 1, the usual existence and uniqueness theorem then extends the interval to $[0, 1]$.
\end{proof}

\begin{lem}\label{s1decmu}
Supppose $\om$ is as in Lemma \ref{ODExist}. Then the smallest zero $s_1$ of the solution $\f$ to (\ref{ODE}) satisfying the Neumann boundary condition at 0 is a decreasing function of $\mu$.
\end{lem}
\begin{proof}
Suppose $\f_i$ solves (\ref{ODE}) with eigenvalue $\mu_i$, together with the Neumann boundary condition at 0. Suppose $\mu_1 > \mu_2$ and $s_1$ is the smallest zero of $\f_1$. Moreover, we normalize $\f_i$ so that $\f_1(0) = \f_2(0)$. We need only to show that $\f_2 > 0$ on $[0, s_1)$.

Suppose not, then there is $x_0 < s_1$ such that $\f_1(x_0) = \f_2(x_0)$, and $\f_1 > 0$ on $[0, x]$. Now the generalized maximum principle (\cite{Protter}, Chapter 1, Section 2, Theorem 5, with $Lu = (\om u')'$, $g(x) = \om'(x)/\om(x)$ bounded below and $h(x) = \mu_1$) applies to show that $\f_2(x) \le \f_1(x)$ on $[0, x_0]$. However, this implies that $\mu_2\f_2(x) \le \mu_1\f_1(x)$ on $[0, x_0]$, so that $(\om(x) \f_2'(x))' \ge (\om(x) \f_1'(x))'$ on $[0, x_0]$. An intergration gives $\om(x) \f_2'(x) \ge \om(x) \f_1'(x)$ on $[0, x_0]$ since $\f_1'(0) = \f_2'(0) = 0$. Since $\om(x) > 0$ when $x > 0$, we know that $\f_2'(x) \ge \f_1'(x)$ on $[0, x_0]$, which in turn implies that $\f_2(x) \ge \f_1(x)$ on $[0, x_0]$. Therefore $\f_2'(x) = \f_1'(x)$ on $[0, x_0]$, but this contradicts the uniqueness proved in Lemma \ref{ODExist}.
\end{proof}

\begin{lem}\label{s1oo0}
Supppose $\om$ is as in Lemma \ref{ODExist}. $s_1(\mu)$ is the smallest zero of the solution to (\ref{ODE}) with $\f(0) = 1$, $\f'(0) = 0$. Then $s_1(0) = +\oo$, $s_1(\mu) \to 0$ as $\mu \to +\oo$.
\end{lem}
\begin{proof}
The first claim is obvious. We now show that $\d := \inf_{\mu \to +\oo} s_1(\mu) = 0$. Suppose not, then $\f(x) > 0$, $\all x < \d$, $\mu > 0$. Consider the solution $\tld \f$ of (\ref{ODE}) with the initial value $\tld \f(\d/2) = 1$, $\tld \f'(\d/2) = 0$. By (\cite{Protter}, Chapter 1, Section 7, Lemma 1, note that (\ref{ODE}) is regular on $(\d/2, \d)$), there is $\mu_0 > 0$ such that $\tld \f$ has a zero $\tld s \in (\d/2, \d)$. On the other hand, the generalized maximum principle applies to show that $\tld\f/\f|_{[\d/2, \tld s]}$ can not have a maximum in $(\d/2, \tld s)$. In addition, since $\tld\f'(\d/2) = 0$ but $\f'(\d/2) = -\rc{\om(\d/2)} \int_0^{\d/2}\mu\f(x)dx < 0$, $(\tld\f/\f)'(\d/2) > 0$, so the maximum is not achieved at $x = \d/2$, nor at $\tld s$ because $\tld\f(s) = 0$. Now the ocntradiction shows that $\d = 0$, or $s_1 \to 0$ as $\mu \to +\oo$.
\end{proof}

\begin{prop}\label{wellpose}
Supppose $\om$ satisfies (a) or (b) in Lemma \ref{ODExist} and one of the two conditions at $x=1$:

\indent (a') $\om(1) > 0$.

\indent (b') $\om(1) = 0$ but $\om(x)$ is decreasing in $x$ for some short interval $[1 - \d, 1]$.

Then there is a smallest $\mu > 0$ and a unique $\f$ satisfying (\ref{ODE}) with Neumann boundary conditions at both 0 and 1 in the classical sense.
\end{prop}
\begin{proof}
Applying Lemmas \ref{s1decmu} and \ref{s1oo0} to $-\om(1-x)$, we know that there is a solution $\tld \f$ of (\ref{ODE}) with the initial value $\tld \f(1) = -1$, $\tld \f'(1) = 0$ with the maximum zero $\tld s_1(\mu)$ an increasing function of $\mu$ rising from $-\oo$ to 1 when $\mu$ varies from 0 to $+\oo$. Therefore we can find $\mu > 0$ such that $s_1(\mu) = \tld s_1(\mu) = s_1$. Since both $\f'(s_1)$ and $\tld \f'(s_1) < 0$, $\f(x)$ and $\f'(s_1) \tld \f(x) /\tld \f'(s_1)$ and their derivatives agree at $x=s_1$. Now by the uniqueness of the solution to the regular ODE, we can piece $\f$ and $\tld\f$ together into a solution satisfying (\ref{ODE}) with Neumann boundary conditions at both 0 and 1. The uniqueness of the eigenfunction has already been covered in Lemma \ref{ODExist}.
\end{proof}

\begin{rem}
It is clear that the cross-sectional volume $\om(x)$ satisfies all the conditions of Proposition \ref{wellpose}.
\end{rem}

\section{Acknowledgements}

This paper is made possible by the Undergraduate Research Opportunity Program of the MIT math department, to which the author wishes to express his thanks. Thanks are also due to Prof. David Jerison for suggesting such an interesting topic and for some very helpful conversations. Finally, it should be mentioned that Section 2 of this paper is a summary of the results in \cite{CJK2009} that the author finds useful.

\bibliographystyle{plain}
\bibliography{nodal}

\begin{thebibliography}{1}

\bibitem{CJK2009}
Sunhi Choi, David Jerison, and Inwon Kim.
\newblock Locating the first nodal set in higher dimensions.
\newblock {\em Transaction of the AMS}, 361(10):5111--5137, 2009.

\bibitem{GJ1995}
Daniel Grieser and David Jerison.
\newblock Asymptotics of the first nodal line.
\newblock {\em Journ\'ees \'Equations aux D\'eriv\'ees Partielles}, pages 1--8,
  1995.

\bibitem{GJ1996}
Daniel Grieser and David Jerison.
\newblock Asymptotics of the first nodal line of a convex domain.
\newblock {\em Inventiones mathematicae}, 125:197--219, 1996.

\bibitem{J1995}
David Jerison.
\newblock The diameter of the first nodal line of a convex domain.
\newblock {\em Annals of Mathematics}, 141(1):1--33, 1995.

\bibitem{J2000}
David Jerison.
\newblock Locating the first nodal line in the neumann problem.
\newblock {\em Transaction of the AMS}, 352(5):2301--2317, 2000.

\bibitem{B1999}
Rodrigo~Ba\ {n}uelos and Krzysztof Burdzy.
\newblock On the `hot spots' conjecture of {J}. {R}auch.
\newblock {\em Journal of Functional Analysis}, 164:1--33, 1999.

\bibitem{Protter}
Murray~H. Protter and Hans~F. Weinberger.
\newblock {\em Maximum Principles in Differential Equations}.
\newblock Springer-Verlag, 1967.

\bibitem{Sperb}
Rene Sperb.
\newblock {\em Maximum Principles and Their Applications}.
\newblock Academic Press, 1981.

\end{thebibliography}
 
\end{document}